\documentclass{amsart}
\usepackage{amssymb, amsmath, amsthm, graphics, comment, xspace, enumerate}
\newtheorem{thm}{Theorem}[section]

\theoremstyle{definition}
\newtheorem{defn}[thm]{Definition}
\newtheorem{example}[thm]{Example}
\theoremstyle{remark}
\newtheorem{rem}[thm]{Remark}
\numberwithin{equation}{section}

\begin{document}
\title[Disjoint distributionally chaotic abstract PDE's]{Disjoint distributionally chaotic abstract PDE's}

\author{Marko Kosti\' c}
\address{Faculty of Technical Sciences,
University of Novi Sad,
Trg D. Obradovi\' ca 6, 21125 Novi Sad, Serbia}
\email{marco.s@verat.net}

{\renewcommand{\thefootnote}{} \footnote{
\\ \text{  }  \ \    {\it Key words and phrases.}   disjoint distributional chaos, disjoint irregular vectors, integrated $C$-semigroups, $\zeta$-times $C$-regularized resolvent families, Fr\' echet spaces
\\  \text{  }  \ \ This research is partially supported by grant 174024 of Ministry
of Science and Technological Development, Republic of Serbia.}}

\begin{abstract}
In this paper, we analyze disjoint distributionally chaotic abstract non-degenerate partial differential equations in Fr\' echet spaces, with integer or Caputo time-fractional derivatives. We present several illustrative examples and applications of our results established. \end{abstract}
\maketitle

\section{Introduction and Preliminaries}

Linear topological dynamics of continuous operators in Banach and Fr\' echet spaces is an extremely popular field of functional analysis. Basic information about this subject can be obtained by consulting the monographs \cite{bayart} by F. Bayart, E. Matheron and
\cite{erdper} by K.-G. Grosse-Erdmann, A. Peris. 

The notion of distributional chaos for interval maps was introduced by B. Schweizer and J. Sm\' ital in \cite{smital} (1994). For linear continuous operators in Banach spaces, 
distributional chaos was firstly considered by J. Duan et al \cite{duan} (1999) and P. Oprocha \cite{p-oprocha} (2006). N. C. Bernardes Jr. et al \cite{2013JFA} (2013) were the first who systematically analyzed distributional chaos for linear continuous operators in Fr\' echet spaces (cf. also the reserach study of J. A. Conejero et al \cite{mendoza} (2016) for a correspoding study of linear not necessarily continuous operators). Some specific properties of distributionally chaotic operators in Banach spaces have been recently  
investigated by  N. C. Bernardes Jr. et al \cite{2018JMAA} (2018).

Disjoint hypercyclic linear operators were introduced independently by L. Bernal--Gonz\'alez \cite{bg07} (2007) and J. B\`es, A. Peris \cite{bp07} (2007). Similar concepts, like disjoint mixing property and disjoint supercyclicity, have been analyzed by a great number of authors after that (for further information about disjoint hypercyclic operators and their generalizations, we refer the reader to \cite{bm201345}, \cite{revista},  \cite{ma10}  and references cited therein.

The main aim of this paper is to continue our recent research study \cite{revista} of disjoint distributional chaos in Fr\' echet spaces by investigating the abstract partial differential equations in Fr\' echet spaces with integer or Caputo time-fractional derivatives (concerning distributional chaos in metric and Fr\' echet spaces, one may refer e.g. to \cite{2011}, \cite{2018JMAA}, \cite{chen-chen}, \cite{marek-trio}, \cite{p-oprocha-tams}-\cite{p-oprocha} and references cited therein). We focus our attention to the analysis of disjoint distributionally chaotic integrated $C$-semigroups, as a rather general concept 
for the investigations of abstract partial differential equations of first order. We also consider disjoint distributionally chaotic properties of abstract time-fractional differential equations with Caputo derivatives; strictly speaking, we analyze disjoint distributional chaos for
$\zeta$-times $C$-regularized resolvent families ($\zeta \in (0,2) \setminus \{1\}$).
For the sake of brevity, we consider only non-degenerate abstract partial differential equations here (for topological dynamics of abstract degenerate  partial differential equations, the reader may consult our joint paper with V. Fedorov \cite{russmath}, the forthcoming monograph \cite{FKP} and references cited therein).

The organization and main ideas of this paper can be described as follows. After giving some necessary explanations about the notation and general framework we are working in, we collect the basic material about integrated $C$-semigroups and $\zeta$-times $C$-regularized resolvent families in two separate subsections, Subsection 1.1 and Subsection 1.2. The second section of paper is devoted to the study of disjoint distributional chaos for integrated $C$-semigroups, while the third section of paper is devoted to the study of 
disjoint distributional chaos for  $\zeta$-times $C$-regularized resolvent families ($\zeta \in (0,2) \setminus \{1\}$). Although not used explicitly, as for single operators \cite{revista}, we also 
provide definitions of disjoint distributinally near to zero vectors, disjoint distributionally unbounded vectors and disjoint distributionally irregular vectors for these solution operator families.
Without any doubt, the main result of paper is Theorem \ref{cvddc}, which provides an efficient tool for proving several other structural results of ours. In addition to the above, a great deal of illustrative examples and applications is presented. 

We use the standard notation in the sequel. By $X$ and $Y$ we denote two
non-trivial Fr\' echet space over the same field of scalars ${\mathbb K}\in \{ {\mathbb R}, {\mathbb C} \}$ and assume that the topologies of $X$ and $Y$ are 
induced by the fundamental systems $(p_{n})_{n\in {\mathbb N}}$ and $(p_{n}^{Y})_{n\in {\mathbb N}}$ of
increasing seminorms, respectively (separability of $X$ and $Y$ will be assumed a priori in future). The translation invariant metric\index{ translation invariant metric} $d :
X\times X \rightarrow [0,\infty),$ defined by
\begin{equation}\label{metri}
d(x,y):=\sum
\limits_{n=1}^{\infty}\frac{1}{2^{n}}\frac{p_{n}(x-y)}{1+p_{n}(x-y)},\
x,\ y\in X,
\end{equation}
satisfies the following properties:
$
d(x+u,y+v)\leq d(x,y)+d(u,v),$ $x,\ y,\ u,\
v\in X;
$
$d(cx,cy)\leq (|c|+1)d(x,y),$ $ c\in {\mathbb K},\ x,\ y\in X,
$
and
$
d(\alpha x,\beta x)\geq \frac{|\alpha-\beta|}{1+|\alpha-\beta|}d(0,x),$ $x\in X,$ $ \alpha,\ \beta \in
{\mathbb K}.$
 Define the translation invariant metric $d_{Y} :
Y\times Y \rightarrow [0,\infty)$ by replacing $p_{n}(\cdot)$ with
$p_{n}^{Y}(\cdot)$ in (\ref{metri}).
If
$(X,\|\cdot \|)$ or $(Y,\|\cdot \|_{Y})$ is a Banach space, then it will be assumed that the
distance of two elements $x,\ y\in X$ ($x,\ y\in Y$) is given by $d(x,y):=\|x-y\|$ ($d_{Y}(x,y):=\|x-y\|_{Y}$).
Keeping in mind this agreement,
our structural results clarified in Fr\' echet spaces retain in the case that $X$ or $Y$ is a Banach space. 

We assume that $N\in {\mathbb N}$ and $N\geq 2.$ Then the fundamental system of increasing seminorms $({\bf p}_{n}^{Y^{N}})_{n\in {\mathbb N}},$ 
where ${\bf p}_{n}^{Y^{N}}(x_{1},\cdot \cdot \cdot,x_{N}):=\sum_{j=1}^{N}p_{n}^{Y}(x_{j}),$ $n\in {\mathbb N}$ ($x_{j}\in Y$ for $1\leq j\leq N$), induces the topology on the Fr\' echet space $Y^{N}.$ The translation invariant metric
\begin{align*}
{\rm d}_{Y^{N}}(\vec{x},\vec{y}):=\sum
\limits_{n=1}^{\infty}\frac{1}{2^{n}}\frac{{\bf p}_{n}(\vec{x}-\vec{y})}{1+{\bf p}_{n}(\vec{x}-\vec{y})},\quad
\vec{x},\ \vec{y}\in Y^{N},
\end{align*}
is strongly equivalent with the metric 
$$
d_{Y^{N}}(\vec{x},\vec{y}):=\max_{1\leq j\leq N}d_{Y}(x_{j},y_{j}),\quad \vec{x}=(x_{1},\cdot \cdot \cdot,x_{N})\in Y^{N},\ \vec{y}=(y_{1},\cdot \cdot \cdot,y_{N})\in Y^{N}.
$$
In the case that $Y$ is a Banach space, then $Y^{N}$ is likewise a Banach space and, in this case, it will be assumed that the distance in $Y^{N}$ is given by $d_{Y^{N}}(\vec{x},\vec{y})=\max_{1\leq j\leq N} \|x_{j}-y_{j}\|_{Y},$ $ \vec{x}\in Y^{N},$ $ \vec{y}\in Y^{N}.$

Suppose that $C\in L(X)$ is injective and $A$ is a closed linear operator with domain and range contained in $X.$
By $D(A),$ $R(A),$ $N(A)$ and $\sigma_{p}(A)$ we denote the domain, range, kernel space and the point spectrum of $A$, respectively.
Set $p_{n}^{C}(x):=p_{n}(C^{-1}x),$ $n\in {\mathbb N},$ $x\in R(C).$ Then
$p_{n}^{C}(\cdot)$ is a seminorm on $R(C)$ and the calibration
$(p_{n}^{C})_{n\in {\mathbb N}}$ induces a Fr\' echet locally convex topology on
$R(C);$ we denote this space simply by $[R(C)].$ Let us recall $[R(C)]$ is separable since $X$ is as well as
that $[R(C)]$ is a Banach space (complex Hilbert space) provided that $X$ is.  Recall that the
$C$-resolvent set of $A,$ denoted by $\rho_{C}(A),$ is defined by
$$
\rho_{C}(A):=\Bigl \{\lambda \in {\mathbb K} : \lambda -A \mbox{
is injective and } (\lambda-A)^{-1}C\in L(X)\Bigr \}.
$$

Set, finally, ${\mathbb C}_{+}:=\{z\in {\mathbb C} : \Re z>0\},$
${\mathbb C}_{-}:=\{z\in {\mathbb C} : \Re z<0\},$
${\mathbb R}_{+}:=(0,\infty),$
${\mathbb R}_{-}:=(-\infty ,0),$
${\mathbb K}_{+}:=\{{\mathbb C}_{+},\ {\mathbb R}_{+}\},$
${\mathbb K}_{-}:=\{{\mathbb C}_{-},\ {\mathbb R}_{-}\},$ 
$\Sigma_{\alpha}:=\{z\in {\mathbb C} : z\neq 0,\ |\arg(z)|<\alpha \}$ ($\alpha \in (0,\pi]$), $\lceil s \rceil:=\inf\{k\in {\mathbb Z} : s\leq k\}$ and
${\mathbb N}_{n}:=\{1,\cdot \cdot \cdot,n\}$ ($s\in {\mathbb R},$ $n\in {\mathbb N}$), $g_{\zeta}(t):=t^{\zeta-1}/\Gamma(\zeta)$ ($t>0$, $\zeta>0$) and recall
that the upper density of a set $D\subseteq [0,\infty)$ is defined
by
$$
\overline{dens}(D):=\limsup_{t\rightarrow
+\infty}\frac{m(D \cap [0,t])}{t},
$$
where $m$ denotes the Lebesgue measure on $[0,\infty).$

We need the following notion from \cite{revista}:

\begin{defn}\label{DC-unbounded-fric-DISJOINT}
Suppose that, for every $j\in {\mathbb N}_{N}$ and $k\in {\mathbb N},$ $A_{j,k} : D(A_{j,k})\subseteq X \rightarrow Y$ is a linear operator and $\tilde{X}$ is a closed linear
subspace of $X.$
Then we say that the sequence $((A_{j,k})_{k\in
{\mathbb N}})_{1\leq j\leq N}$ is disjoint
$\tilde{X}$-distributionally chaotic, $(d,\tilde{X})$-distributionally chaotic in short, iff there exist an uncountable
set\\ $S\subseteq \bigcap_{j=1}^{N} \bigcap_{k=1}^{\infty} D(A_{j,k}) \cap \tilde{X}$ and
$\sigma>0$ such that for each $\epsilon>0$ and for each pair $x,\
y\in S$ of distinct points we have 
\begin{align*}
\begin{split}
& \overline{dens}\Biggl( \bigcap_{j\in {\mathbb N}_{N}} \bigl\{k \in {\mathbb N} :
d_{Y}\bigl(A_{j,k}x,A_{j,k}y\bigr)\geq \sigma \bigr\}\Biggr)=1,\mbox{ and }
\\ 
& \overline{dens}\Biggl( \bigcap_{j\in {\mathbb N}_{N}} \bigl\{k \in {\mathbb N} : d_{Y}\bigl(A_{j,k}x,A_{j,k}y\bigr)
<\epsilon \bigr\}\Biggr)=1.
\end{split}
\end{align*}
The sequence $((A_{j,k})_{k\in
{\mathbb N}})_{1\leq j\leq N}$ is said to be densely
$(d,\tilde{X})$-distributionally chaotic iff $S$ can be chosen to be dense in $\tilde{X}.$
A finite sequence $(A_{j})_{1\leq j\leq N}$ of closed linear operators on $X$ is said to be (densely)
$\tilde{X}$-distributionally chaotic\index{$\tilde{X}$-distributionally chaotic operator} iff the sequence $((A_{j,k}\equiv
A_{j}^{k})_{k\in {\mathbb N}})_{1\leq j\leq N}$ is.
The set $S$ is said to be $(d,\sigma_{\tilde{X}})$-scrambled set\index{ $\sigma_{\tilde{X}}$-scrambled set} ($(d,\sigma)$-scrambled set in the case\index{$\sigma$-scrambled set}
that $\tilde{X}=X$) of $((A_{j,k})_{k\in
{\mathbb N}})_{1\leq j\leq N}$ ($(A_{j})_{1\leq j\leq N}$);  in the case that
$\tilde{X}=X,$ then we also say that the sequence $((A_{j,k})_{k\in
{\mathbb N}})_{1\leq j\leq N}$ ($(A_{j})_{1\leq j\leq N}$) is disjoint distributionally chaotic, $d$-distributionally chaotic in short.
\end{defn}

\subsection{Integrated $C$-semigroups}\label{kretinjo-zero}

The following definition is fundamental in the theory of abstract ill-posed differential equations of first order (cf. \cite{knjigah}-\cite{knjigaho} for more details on the subject): 

\begin{defn}\label{first}
Suppose that $\alpha \geq 0$ and $A$ is a closed linear operator. If
there exists a strongly continuous operator family
$(S_\alpha(t))_{t\geq 0}\subseteq L(X)$ such that:
\begin{itemize}
\item[(i)] $S_\alpha(t)A\subseteq AS_\alpha(t)$, $t\geq 0$,
\item[(ii)] $S_\alpha(t)C=CS_\alpha(t)$, $t\geq 0$,
\item[(iii)] for all $x\in X$ and $t\geq 0$: $\int_0^tS_\alpha(s)x\,ds\in D(A)$ and
\begin{align*}
A\int\limits_0^tS_\alpha(s)x\,ds=S_\alpha(t)x-g_{\alpha +1}(t)Cx,
\end{align*}
\end{itemize}
then it is said that $A$ is a subgenerator of a (global)
$\alpha$-times integrated $C$-semigroup $(S_\alpha(t))_{t\geq 0}$.
\end{defn}

If $\alpha =0,$ then $(S_0(t))_{t\geq 0}$ is
also said to be a $C$-regularized semigroup with subgenerator $A$ (we refer the reader to \cite{knjigaho} for definition of an entire $C$-regularized group and its integral generator (subgenerator)). 
The integral generator of $(S_\alpha(t))_{t\geq 0}$ is defined by
\begin{align*}
\hat{A}:=\Biggl\{(x,y)\in X\times
X:S_\alpha(t)x-g_{\alpha+1}(t)Cx=\int\limits^t_0S_\alpha(s)y\,ds,\;t\geq
0\Biggr\}.
\end{align*}
Let us recall that the integral generator of $(S_\alpha(t))_{t\geq 0}$
is a closed linear operator which extends any subgenerator of $(S_\alpha(t))_{t\geq 0}.$
Furthermore, for any subgenerator $A$ of $(S_\alpha(t))_{t\geq 0},$ the
following equality holds $\hat{A}=C^{-1}AC.$

Denote by $Z_{1}(A)$ the space consisting of
those elements $x\in X$ for which there exists a unique $X$-valued
continuous mapping satisfying $\int^t_0 u(s,x)\,ds\in D(A)$ and
$A\int^t_0 u(s,x)\,ds=u(t,x)-x$, $t\geq 0,$ i.e., the unique mild
solution of the corresponding Cauchy problem $(ACP_{1}):$
$$
(ACP_{1}) : u^{\prime}(t)=Au(t),\ t\geq 0, \ u(0)=x.
$$
If $A$ is a subgenerator (the integral generator) of a
global $\alpha$-times integrated $C$-semigroup $(S_\alpha(t))_{t\geq
0},$ then there is only one (trivial) mild solution of $
(ACP_{1})$ with $x=0,$ so that $Z_{1}(A)$ is a linear subspace of $X.$
Moreover, for every number $\beta>\alpha ,$ the
operator $A$ is a subgenerator (the integral generator) of a global
$\beta$-times integrated $C$-semigroup $(S_\beta(t)\equiv
(g_{\beta-\alpha}\ast S_{\alpha}\cdot)(t))_{t\geq 0}.$ As it is well known, the
space
$Z_{1}(A)$
consists exactly of those elements $x\in X$ for which the mapping
$t\mapsto C^{-1}S_{\lceil \alpha \rceil}(t)x,$ $t\geq 0$ is well
defined and $\lceil \alpha \rceil$-times continuously differentiable
on $[0,\infty);$ see e.g. \cite{knjigaho}. As it is usually done in the  theory of $C$-distribution semigroups, we set
$$
{\mathcal G}(\varphi)x:=(-1)^{\lceil \alpha \rceil}\int
\limits^{\infty}_{0}\varphi^{(\lceil \alpha \rceil)}(t)S_{\lceil
\alpha \rceil}(t)x\, dt,\quad \varphi \in {\mathcal D}_{{\mathbb
K}},\ x\in X
$$
and
$$
G\bigl(\delta_{t}\bigr)x:=\frac{d^{\lceil \alpha \rceil}}{dt^{\lceil
\alpha \rceil}}C^{-1}S_{\lceil \alpha \rceil}(t)x, \quad t\geq 0,\
x\in Z_{1}(A);
$$
here ${\mathcal D}_{{\mathbb
K}}$ denotes the space of ${\mathbb K}$-valued smooth test functions with compact support contained in $K.$
Then the following
holds: $G(\delta_t)(Z_{1}(A))\subseteq Z_{1}(A),\ t\geq 0,$ $G(\delta_t)C\subseteq CG(\delta_t),$ $t\geq 0$
and
\begin{equation}\label{C-DS}
G\bigl(\delta_s\bigr)G\bigl(\delta_t
\bigr)x=G\bigl(\delta_{t+s}\bigr)x, \;t,\,s\geq 0,\ x\in Z_{1}(A).
\end{equation}
Is is also known that the solution
space $Z_{1}(A)$ is independent of the choice of
$(S_\alpha(t))_{t\geq 0}$ in the
following sense: If $C_{1}\in L(X)$ is another injective operator
with $C_{1}A\subseteq AC_{1},$ $\gamma \geq 0,$ $x\in X$ and $A$ is
a subgenerator (the integral generator) of a global $\gamma$-times
integrated $C_{1}$-semigroup $(S^\gamma(t))_{t\geq 0},$ then the mapping $t\mapsto
C^{-1}S_{\lceil \alpha \rceil}(t)x,$ $t\geq 0$ is well defined and
$\lceil \alpha \rceil$-times continuously differentiable on
$[0,\infty)$ iff the mapping $t\mapsto C_{1}^{-1}S^{\lceil \gamma
\rceil}(t)x,$ $t\geq 0$ is well defined and $\lceil \gamma
\rceil$-times continuously differentiable on $[0,\infty)$. In this case, we have 
$u(t;x):=G(\delta_{t})x=\frac{d^{\lceil \gamma \rceil}}{dt^{\lceil
\gamma \rceil}}C_{1}^{-1}S^{\lceil \gamma \rceil}(t)x,$ $t\geq 0$
is a unique mild
solution of the corresponding Cauchy problem $(ACP_{1}).$

The notions of exponential equicontinuity and analyticity of integrated $C$-semigroups are well known; the basic results about integrated $C$-cosine functions can be found in \cite{knjigah}-\cite{knjigaho}, as well.

\subsection{$\zeta$-Times $C$-regularized resolvent families ($\zeta \in (0,2) \setminus \{1\}$)}\label{kretinjo}

The following definition has been introduced by M. Li, Q. Zheng and J. Zhang in \cite{zhengtaiwan} (see \cite{knjigah}-\cite{FKP} for more details about abstract time-fractional differential equations):

\begin{defn}\label{1.1}
Suppose that $\zeta >0$ and $A$ is a closed linear operator on $X$. A
strongly continuous operator family $(R_{\zeta}(t))_{t\geq 0}$ is
said to be a $\zeta$-times $C$-regularized resolvent family
having $A$ as a subgenerator iff the following holds:
\begin{itemize}
\item[(i)] $R_{\zeta}(t)A\subseteq AR_{\zeta}(t),\ t\geq 0,$ $R_{\zeta}(0)=C$ and
$CA\subseteq AC,$
\item[(ii)] $R_{\zeta}(t)C=CR_{\zeta}(t),\ t\geq 0$ and
\item[(iii)] $R_{\zeta}(t)x=Cx+\int_{0}^{t}g_{\zeta}(t-s)
AR_{\zeta}(s)x\, ds,\ t\geq 0,\ x\in D(A).$
\end{itemize}
In the case $C=I,$ then we also say that
$(R_{\zeta}(t))_{t\geq 0}$ is a 
$\zeta$-times regularized resolvent family with subgenerator $A.$
\end{defn}

The integral
generator of $(R_{\zeta}(t))_{t\geq 0}$ is defined by
$$
\hat{A}:=\Biggl\{(x,y)\in X \times X : R_{\zeta}(t)x-Cx=\int
\limits_{0}^{t}g_{\zeta}(t-s)R_{\zeta}(s)y\, ds
\mbox{ for all } t\geq 0\Biggr\},
$$
and it is a closed linear operator which extends any
subgenerator of $(R_{\zeta}(t))_{t\geq 0}.$ 

Let $m:=\lceil \zeta \rceil.$ The Caputo fractional derivative\index{fractional derivatives!Caputo}
${\mathbf D}_{t}^{\zeta}u(t)$ is defined for those functions $u\in
C^{m-1}([0,\infty) : X)$ for which $g_{m-\zeta} \ast
(u-\sum_{k=0}^{m-1}u_{k}g_{k+1}) \in C^{m}([0,\infty) : X),$
by
$$
{\mathbf
D}_{t}^{\zeta}u(t):=\frac{d^{m}}{dt^{m}}\Biggl[g_{m-\zeta}
\ast \Biggl(u-\sum_{k=0}^{m-1}u_{k}g_{k+1}\Biggl)\Biggr].
$$
The abstract evolution equation
\begin{align}\label{bez}
{\mathbf D}_{t}^{\zeta}u(t)=Au(t),\ t>0,\ u(0)=x,\ u^{(k)}(0)=0,\
k=1,\cdot \cdot \cdot, m-1,
\end{align}
is well posed in the sense of \cite[Definition 2.2]{bajlekova} iff the abstract Volterra equation
\begin{equation}\label{prus}
u(t;x)=x+\int
\limits^{t}_{0}g_{\zeta}(t-s)Au(s;x)\, ds,\
t\geq 0,
\end{equation}
is well posed in the sense of \cite[Definition 1.2]{prus}.
 Suppose that $A$ is a subgenerator of an
$\zeta$-times $C$-regularized resolvent family
$(R_{\zeta}(t))_{t\geq 0},$ and
\begin{equation}\label{ceregulari}
R_{\zeta}(t)x=Cx+A\int \limits_{0}^{t}g_{\zeta}(t-s)
R_{\zeta}(s)x\, ds,\ t\geq 0,\ x\in X.
\end{equation}
Denote by $Z_{\zeta}(A)$ the set
consisting of those vectors $x\in X$ such that
$R_{\zeta}(t)x\in R(C),$ $t\geq 0$ and the mapping $t\mapsto
C^{-1}R_{\zeta}(t)x,$ $t\geq 0$ is continuous. Then $R(C)\subseteq
Z_{\zeta}(A),$ and $x\in Z_{\zeta}(A)$ iff there exists a unique strong
solution of (\ref{prus}); if this is the case, the unique strong
solution of (\ref{prus}) is given by $u(t;x)=C^{-1}R_{\zeta}(t)x,$
$t\geq 0.$ In the sequel, we assume the validity of \eqref{ceregulari} a priori.

Denote by $E_{\beta}(z)$ the Mittag-Leffler
function $E_{\beta}(z):=\sum
 _{n=0}^{\infty}\frac{z^{n}}{\Gamma(\beta n+1)},$ $z\in
{\mathbb C},$ where $\beta>0.$ 
Suppose, further, that $\zeta \in (0,2) \setminus \{1\}$ and $l\in {\mathbb N}
\setminus \{1\}.$ We will use the following asymptotic formulae for the
Mittag-Leffler functions (\cite{bajlekova}):
\begin{equation}\label{asim1}
E_{\zeta}(z)=
\frac{1}{\zeta}e^{z^{1/\zeta}}+\varepsilon_{\zeta}(z),\ |\arg
(z)|<\zeta \pi/2,
\end{equation}
and
\begin{equation}\label{asim2}
E_{\zeta}(z)= \varepsilon_{\zeta}(z),\ |\arg (-z)|<\pi-\zeta
\pi/2,
\end{equation}
where
\begin{equation}\label{asim3}
\varepsilon_{\zeta}(z)=\sum
\limits_{n=1}^{l-1}\frac{z^{-n}}{\Gamma(1-\zeta n)}+O(|z|^{-l}),\
|z|\rightarrow \infty .
\end{equation}

\section{Disjoint distributionally chaotic properties of abstract PDEs of first order}\label{d-MLOs}

In \cite{revista}, we have introduced and analyzed twelve different types of disjoint distributional chaos for multivalued linear operators in Fr\' echet spaces. For the sake of simplicity, we will consider here only one type of 
disjoint distributional chaos for integrated $C$-semigroups, disjoint distributional chaos of type $1.$ This is the most intriguing type of disjoint distributional chaos considered in \cite{revista} because it is the strongest one and implies all others (we will not particularly emphasize further that this is disjoint distributional chaos of type $1$): 

\begin{defn}\label{to fuck}
Let $\alpha_{j}\geq 0,$ let $C_{j}\in L(X)$ be injective for all $j\in {\mathbb N}_{N}$  and let $(S_{\alpha_{j}}(t))_{t\geq
0}$ be a global $\alpha_{j}$-times integrated $C_{j}$-semigroup with the integral generator $A_{j}$ ($j\in {\mathbb N}_{N}$). 
Suppose that $\tilde{X}$ is a closed linear
subspace of $X.$ Denote by $t\mapsto G_{j}(\delta_{t})x,$ $t\geq 0$ the unique mild
solution of the corresponding Cauchy problem $(ACP_{1}),$ with the operator $A$ replaced by $A_{j}$ therein ($j\in {\mathbb N}_{N}$). Then we say that $((S_{\alpha_{j}}(t))_{t\geq
0})_{1\leq j\leq N}$ are
disjoint
$\tilde{X}$-distributionally chaotic, $(d,\tilde{X})$-distributionally chaotic in short, iff there exist an uncountable
set $S\subseteq \bigcap_{j=1}^{N} Z_{1}(A_{j}) \cap \tilde{X}$ and
$\sigma>0$ such that for each $\epsilon>0$ and for each pair $x,\
y\in S$ of distinct points we have that for each $j\in {\mathbb N}_{N}$ and $t\geq 0$ 
we have that
\begin{align*}
\begin{split}
& \overline{dens}\Biggl( \bigcap_{j\in {\mathbb N}_{N}} \bigl\{t\geq 0 :
d_{Y}\bigl(G_{j}(\delta_{t})x,G_{j}(\delta_{t})y\bigr)\geq \sigma \bigr\}\Biggr)=1,\mbox{ and }
\\ 
& \overline{dens}\Biggl( \bigcap_{j\in {\mathbb N}_{N}} \bigl\{t\geq 0 : d_{Y}\bigl(G_{j}(\delta_{t})x,G_{j}(\delta_{t})y\bigr)
<\epsilon \bigr\}\Biggr)=1.
\end{split}
\end{align*}

The sequence $(S_{\alpha_{j}}(t))_{t\geq
0}$ is said to be densely
$(d,\tilde{X})$-distributionally chaotic iff $S$ can be chosen to be dense in $\tilde{X}.$
The set $S$ is said to be $(d,\sigma_{\tilde{X}})$-scrambled set\index{ $\sigma_{\tilde{X}}$-scrambled set} ($(d,\sigma)$-scrambled set in the case\index{$\sigma$-scrambled set}
that $\tilde{X}=X$) of $((S_{\alpha_{j}}(t))_{t\geq
0})_{1\leq j\leq N}$;  in the case that
$\tilde{X}=X,$ then we also say that the sequence $((S_{\alpha_{j}}(t))_{t\geq
0})_{1\leq j\leq N}$ is (densely) disjoint distributionally chaotic, (densely) $d$-distributionally chaotic in short.
\end{defn}

Now we introduce the notion of disjoint distributionally irregular vectors for integrated $C$-semigroups:

\begin{defn}\label{dccfa}
Let $\alpha_{j}\geq 0,$ let $C_{j}\in L(X)$ be injective for all $j\in {\mathbb N}_{N}$  and let $(S_{\alpha_{j}}(t))_{t\geq
0}$ be a global $\alpha_{j}$-times integrated $C_{j}$-semigroup with the integral generator $A_{j}$ ($j\in {\mathbb N}_{N}$). Suppose that $\tilde{X}$ is a closed linear
subspace of $X,$ $m\in {\mathbb N}$ and
$x\in  \bigcap_{j=1}^{N} Z_{1}(A_{j})\cap \tilde{X}.$ Denote by $t\mapsto G_{j}(\delta_{t})x,$ $t\geq 0$ the unique mild
solution of the corresponding Cauchy problem $(ACP_{1}),$  with the operator $A$ replaced by $A_{j}$ therein ($j\in {\mathbb N}_{N}$). Then we say that:
\begin{itemize}
\item[(i)] $x$ is disjoint distributionally near to $0$ for 
$((S_{\alpha_{j}}(t))_{t\geq
0})_{1\leq j\leq N}$ iff
there exists $A\subseteq [0,\infty)$ such that $\overline{Dens}(A)=1$ and
$\lim_{s\rightarrow \infty,s\in A}G_{j}(\delta_{s})x=0$ for all $j\in {\mathbb N}_{N};$ 
\item[(ii)] $x$ is disjoint distributionally $m$-unbounded for 
$((S_{\alpha_{j}}(t))_{t\geq
0})_{1\leq j\leq N}$ iff
there exists $B\subseteq [0,\infty)$ such that $\overline{Dens}(B)=1$ and
$\lim_{s\rightarrow \infty,s\in B}p_{m}(G_{j}(\delta_{s})x)=0$ for all $j\in {\mathbb N}_{N};$ $x$ is
disjoint distributionally unbounded for the tuple
$((S_{\alpha_{j}}(t))_{t\geq
0})_{1\leq j\leq N}$ iff there exists $q\in {\mathbb N}$ such that
$x$ is disjoint distributionally $q$-unbounded for 
$((S_{\alpha_{j}}(t))_{t\geq
0})_{1\leq j\leq N}$;
\item[(iii)] $x$ is a disjoint $\tilde{X}$-distributionally irregular vector for
$((S_{\alpha_{j}}(t))_{t\geq
0})_{1\leq j\leq N}$
(disjoint distributionally irregular vector for 
$((S_{\alpha_{j}}(t))_{t\geq
0})_{1\leq j\leq N}$ simply, in the case that
$\tilde{X}=X$) iff $x$ is both disjoint distributionally near to $0$ and
disjoint distributionally unbounded.
\end{itemize}
\end{defn}

The following important result is a continuous analogue of \cite[Theorem 4.3]{revista}. It also provides an extension of \cite[Theorem 4.1]{mendoza} for disjoint distributional chaos:

\begin{thm}\label{cvddc}
Suppose that $X_{0}$ is a dense linear subspace of
$X,$ $(T_{j}(t))_{t\geq 0}\subseteq L(X,Y)$ is a strongly continuous operator family for each $j\in {\mathbb N}_{N},$
as well as:
\begin{itemize}
\item[(a)] $\lim_{t\rightarrow \infty}T_{j}(t)x=0,$ $x\in X_{0},$ $j\in {\mathbb N}_{N},$
\item[(b)] there exist $x\in X,$ $m\in {\mathbb N}$ and a set $B\subseteq [0,\infty)$ such that
$\overline{Dens}(B)=1,$ and $\lim_{t\rightarrow \infty ,t\in B}p_{m}(T_{j}(t)x)=\infty$  for each $j\in {\mathbb N}_{N},$ resp.\\
$\lim_{t\rightarrow \infty ,t\in B}\|T_{j}(t)x\|=\infty$ for each $j\in {\mathbb N}_{N},$ if $X$ is a Banach space.
\end{itemize}
Then there exist a dense linear subspace
$S$ of $X$ and a number $\sigma>0$ such that
for each $\epsilon>0$
and for each pair $x,\ y\in S$
of distinct points we have that
$$
\overline{Dens}\Biggl( \bigcap_{j\in {\mathbb N}_{N}} \Bigl\{ s\geq 0 : d_{Y}\bigl(T_{j}(s)x,T_{j}(s)y\bigr)\geq \sigma  \Bigr\}\Biggr)=1
$$
and
$$
\overline{Dens}\Biggl( \bigcap_{j\in {\mathbb N}_{N}} \Bigl\{ s\geq 0 : d_{Y}\bigl(T_{j}(s)x,T_{j}(s)y\bigr)< \epsilon  \Bigr\}\Biggr)=1.
$$
\end{thm}

\begin{proof}
The proof is very similar to those of \cite[Theorem 15]{2013JFA} and \cite[Theorem 4.1]{mendoza}, so that we will only outline the main points of the proof.
Consider first the case in which $X$ and $Y$ are Frech\'et spaces.
If so, the family $(T_{j}(t))_{t\geq 0}\subseteq L(X,Y)$
is locally equicontinuous for all $j\in {\mathbb N}_{N}.$ Hence, for
every $l,\ n\in {\mathbb N},
$ there exist $c_{l,n}>0$ and $a_{l,n}\in {\mathbb N}$ such that $p_{l}^{Y}(T_{j}(t)x)\leq c_{l,n}p_{a_{l,n}}(x),$ $x\in X,$ $t\in [0,n],$ $j\in {\mathbb N}_{N}.$
Suppose, for the time being, that:
\begin{equation}\label{saban}
p_{k}^{Y}(T_{j}(t)x)\leq p_{k+\lceil t \rceil}(x),\quad x\in X,\ t\geq 0,\ k\in {\mathbb N},\ j\in {\mathbb N}_{N}.
\end{equation}
We may assume that
$m=1.$
Then there exist a sequence $(x_{k})_{k\in {\mathbb N}}$ in $X_{0}$ such
that $p_{k}(x_{k})\leq 1,$ $k\in {\mathbb N}$ and a strictly increasing sequence
of positive real numbers $(t_{k})_{k\in {\mathbb N}}$ tending to infinity such that:
$$
\overline{Dens}\Bigl( \Bigl\{ 1\leq s \leq t_{k} : p_{1}\bigl(T_{j}(s)x_{k}\bigr) >k2^{k} \Bigr \} \Bigr)\geq t_{k}\Bigl(1-k^{-2} \Bigr)
$$
and
$$
\overline{Dens}\Bigl( \Bigl\{ 1\leq s \leq t_{k} : p_{k}\bigl(T_{j}(s)x_{l}\bigr) < k^{-1} \Bigr \} \Bigr)\geq t_{k}\Bigl(1-k^{-2} \Bigr),\quad
l=1,\cdot \cdot \cdot,k-1,
$$
for any $j\in {\mathbb N}_{N}.$
Further on, it is clear that there is a strictly increasing sequence $(r_{s})_{s\in {\mathbb N}}$ of positive integers satisfying that:
$$
r_{s+1}\geq 1+r_{s}+\lceil t_{r_{s}+1} \rceil,\quad s\in {\mathbb N}.
$$
Arguing as in \cite[Theorem 15]{2013JFA},
we get that there exists a dense linear subspace $S$ of $X$ such that, for every $x\in S,$ there exist two sets $A_{x},\ B_{x}\subseteq [0,\infty)$
such that $\overline{Dens}(A)=\overline{Dens}(B)=1,$ $\lim_{t\rightarrow \infty,t\in A_{x}}T_{j}(t)x=0$ and
$\lim_{t\rightarrow \infty,t\in B_{x}}p_{1}(T_{j}(t)x)=\infty .$ Now the final conlusion of theorem follows as in the discrete case. Finally, a few words about the process of renorming.
Introducing recursively the following fundamental system of increasing seminorms $p_{n}'(\cdot)$ ($n\in {\mathbb N}$) on $X:$
\begin{align*}
& p_{1}'(x)\equiv  p_{1}(x),\quad x\in X,\\
& p_{2}'(x)\equiv p_{1}'(x)+c_{1,1}p_{a_{1,1}}(x)+p_{2}(x),\quad x\in X,\\
& \cdot \cdot \cdot \\
& p_{n+1}'(x)\equiv p_{n}'(x)+c_{1,n}p_{a_{1,n}}(x)+\cdot \cdot \cdot +c_{n,1}p_{a_{n,1}}(x)+p_{n+1}(x),\quad x\in X,\\
& \cdot \cdot \cdot ,
\end{align*}
we may assume without loss of generality that (\ref{saban}) holds; hence, the assertion is proved in the case that $X$ and $Y$ are Frech\'et spaces.
If $X$ or $Y$ is a Banach space, say $Y$,
then we can `renorm' it, by endowing $Y$ with the following
increasing family of seminorms $p_{n}^{Y}(y):=n\|y\|_{Y}$ ($n\in {\mathbb N},$
$y\in Y$), which turns the space $Y$ into a linearly and topologically homeomorphic
Fr\' echet space. This completes the proof of theorem. 
\end{proof}

It is clear that Theorem \ref{cvddc} can be directly applied to strongly continuous semigroups and thus provides a great number of concrete examples of disjoint distributionally chaotic single operators (see \cite[Chapter 3]{knjigaho} and \cite{mendozaa} for more details). In the remaining part of paper, we will primarily examine possible applications of Theorem \ref{cvddc} to the abstract ill-posed abstract Cauchy problems. 

For any injective operator $C\in L(X),$ any closed linear operator $A$ commuting with $C$ and any positive integer $n\in {\mathbb N},$ 
we endow the space $C(D(A^{n}))$ with the following family of seminorms
$p_{m,n}(Cx):=p_{m}(x)+p_{m}(Ax)+\cdot \cdot \cdot +p_{m}(A^{n}x),$ $m\in {\mathbb N},$ $x\in D(A^{n})$ ($n\in {\mathbb N}$). Of course,
if $X$ is a Banach space, then the space $C(D(A^{n}))$ carries the topology induced by the norm
$\|Cx\|_{n}:=\|x\|+\|Ax\|+\cdot \cdot \cdot +\|A^{n}x\|,$ $x\in D(A^{n}).$ Denote
this space by $[C(D(A^{n}))].$ 

Now we will reconsider the assertion of \cite[Theorem 5.4]{mendoza} for disjoint distributional chaos:

\begin{thm}\label{teoremica}
Suppose that $\alpha_{j}\geq 0,$ $t_{j}>0$ and $A_{j}$ subgenerates a global $\alpha_{j}$-times integrated $C_{j}$-semigroup $(S_{\alpha_{j}}(t))_{t\geq 0}$ on $X$ ($j\in {\mathbb N}_{N}$).
Let $n_{j}:= \lceil \alpha_{j} \rceil$ for any $j\in {\mathbb N}_{N},$ let 
$C\in L(X)$ be injective, and let $[R(C)]$ be continuously embedded in the space
$[C_{j}(D(A_{j}^{n_{j}}))]$ for all $j\in {\mathbb N}_{N}.$ Furthermore, suppose that the following conditions hold:
\begin{itemize}
\item[(i)] There exists a dense subset $X_{0}'$ of $[R(C)]$ such that $\lim_{t\rightarrow \infty}G_{j}(\delta_{t})x=0,$ $x\in X_{0}',$ $j\in {\mathbb N}_{N}.$
\item[(ii)] There exist $x\in R(C)$ and $m\in {\mathbb N}$ such that 
$\lim_{t\rightarrow \infty}p_{m}(G_{j}(\delta_{t})x)=\infty,$ $j\in {\mathbb N}_{N}$
($\lim_{t\rightarrow \infty}\|G_{j}(\delta_{t})x\|=\infty ,$ $j\in {\mathbb N}_{N}$ in the case that $X$ is a Banach space).
\end{itemize}
Then $((S_{\alpha_{j}}(t))_{t\geq 0})_{1\leq j\leq N}$ and the operators $G_{1}(\delta_{t_{1}}), \, G_{2}(\delta_{t_{2}}),\cdot \cdot \cdot ,\, G_{N}(\delta_{t_{N}})$ are disjoint distributionally chaotic; if $R(C)$ is dense in $X,$
then $((S_{\alpha_{j}}(t))_{t\geq 0})_{1\leq j\leq N}$ and the operators $G_{1}(\delta_{t_{1}}), \, G_{2}(\delta_{t_{2}}),\cdot \cdot \cdot ,\, G_{N}(\delta_{t_{N}})$ are densely disjoint distributionally chaotic.
\end{thm}

\begin{proof}
It is clear that $[R(C)]$ is separable. 
Let us recall that $C_{j}(D(A_{j}^{n_{j}}))\subseteq Z_{1}(A_{j})$ for all $j\in {\mathbb N}_{N};$ furthermore, if $x=C_{j}y\in C_{j}(D(A_{j}^{n_{j}})),$ then for every $t\geq 0$ we have:
\begin{align*}
G_{j}\bigl( \delta_{t}  \bigr)x=S_{\alpha_{j}}(t)A_{j}^{n_{j}}y+\sum \limits_{i=0}^{n_{j}-1}\frac{t^{n_{j}-i-1}}{(n_{j}-i-1)!}C_{j}A^{n_{j}-1-i}y,\quad  j\in {\mathbb N}_{N}.
\end{align*}
Since $[R(C)]$ is continuously embedded in the space
$[C_{j}(D(A_{j}^{n_{j}}))]$ for all $j\in {\mathbb N}_{N},$
we have that, for every $t\geq 0,$ the mapping $G(\delta_{t}) : [R(C)] \rightarrow X$ is linear and continuous.
Furthermore, the family $(G(\delta_{t}))_{t\geq 0}\subseteq L([R(C)] , X)$ is strongly continuous.
We define $T_{j,k}\equiv G(\delta_{kt_{j}}): [R(C)] \rightarrow X$ ($ j\in {\mathbb N}_{N},$ $k\in {\mathbb N}$). Then $((T_{j,k})_{k\in {\mathbb N}})_{1\leq j\leq N}\subseteq  L([R(C)],X)$ and (\ref{C-DS}) yields that $T_{j,k}x=G_{j}(\delta_{t_{j}})^{k}x,$ $x\in R(C).$ 
Now an application of \cite[Theorem 4.4]{revista} yields that
the operators $G_{1}(\delta_{t_{1}}), \, G_{2}(\delta_{t_{2}}),\cdot \cdot \cdot ,\, G_{N}(\delta_{t_{N}})$ are disjoint distributionally chaotic, while an application of Theorem \ref{cvddc} yields that $((S_{\alpha_{j}}(t))_{t\geq 0})_{1\leq j\leq N}$ are disjoint distributionally chaotic. Finally, if
$R(C)$ is dense in $X,$ then it almost trivially follows from the foregoing that $((S_{\alpha_{j}}(t))_{t\geq 0})_{1\leq j\leq N}$ and the operators $G_{1}(\delta_{t_{1}}), \, G_{2}(\delta_{t_{2}}),\cdot \cdot \cdot ,\, G_{N}(\delta_{t_{N}})$ are densely disjoint distributionally chaotic.
\end{proof}

\begin{rem}\label{dusty}
\begin{itemize}
\item[(i)] If $\lambda_{j}\in \rho_{C_{j}}(A_{j})$ for all $j\in {\mathbb N}_{N},$ then the choice $C:=\prod_{j=1}^{N}C_{j}((\lambda_{j}-A_{j})^{-1}C_{j})^{n}$ can be always made.
\item[(ii)] If $\lambda_{j} \in \sigma_{p}(A_{j})$ and $A_{j}x=\lambda_{j} x$ for some $x\in X \setminus \{0\}$ and $j\in {\mathbb N}_{N},$ then $x\in Z_{1}(A_{j})$
and $G_{j}(\delta_{t})x=e^{\lambda_{j} t}x,$ $t\geq 0.$ In particular,
$\lim_{t\rightarrow \infty}G_{j}(\delta_{t})x=0$ if $\lambda \in {\mathbb K}_{-},$ and there exists
$m\in {\mathbb N}$ such that $\lim_{t\rightarrow \infty}p_{m}(G_{j}(\delta_{t})x)=\infty$
($\lim_{t\rightarrow \infty}\|G_{j}(\delta_{t})x\|=\infty$ in the case that $X$ is a Banach space),
if $\lambda \in {\mathbb K}_{+}$. 
\item[(iii)] Assume now that all requirements of Theorem \ref{teoremica} stated before the formulation of (i)-(iii) hold true, as well as that $X_{0}:=\{Cx : (\forall j\in {\mathbb N}_{N})\, (\exists \lambda_{j,-}\in {\mathbb K}_{-})\, A_{j}Cx=\lambda_{j,-}Cx\}.$ Suppose that
\begin{itemize}
\item[(a)] $\tilde{X}:=\overline{X_{0}}^{[R(C)]}$ is non-trivial subspace of $[R(C)],$ and
\item[(b)] there exist a vector $Cx\in \tilde{X}$ and the scalars $\lambda_{j,+}\in {\mathbb K}_{+}$ such that $A_{j}Cx=\lambda_{j,+}Cx$ for all $j\in {\mathbb N}_{N}.$
\end{itemize}
Repeating literally the arguments given in the proof of Theorem \ref{teoremica}, with the spaces $[R(C)]$ and $X_{0}'$ replaced with the spaces $\tilde{X}$ and $X_{0}$ therein, we get that $((S_{\alpha_{j}}(t))_{t\geq 0})_{1\leq j\leq N}$ and the operators $G_{1}(\delta_{t_{1}}), \, G_{2}(\delta_{t_{2}}),\cdot \cdot \cdot ,\, G_{N}(\delta_{t_{N}})$ are disjoint $\tilde{X}$-distributionally chaotic. The question whether we can make a choice $\tilde{X}=R(C)$ has an affirmative answer in the case that the operators $A_{j}$
have nice supplies of eigenfunctions (see e.g. the proof of Desch-Schappacher-Webb criterion for strongly continuous semigroups \cite[Theorem 3.1]{fund},
as well as Example \ref{polinomi} below).
\end{itemize}
\end{rem}

The interested reader may try to formulate an analogue of \cite[Theorem 5.9]{mendoza} for disjoint distributional chaos of entire $C$-regularized groups. 

We proceed by providing two illustrative examples.

\begin{example}\label{polinomi} (\cite{ralphchaos})
Assume that ${\mathbb K}={\mathbb C},$
$\omega_{1},$ $\omega_{2},$ $V_{\omega_{2},\omega_{1}},$ $Q_{j}(z),$
$Q_{j}(B),$ $h_{\mu}$ and $X$ possess the same meaning as in
\cite[Section 5]{ralphchaos}, as well as that the number $L\in {\mathbb N}$ is sufficiently large and takes the role of number $N$ from this section. Let $t_{j}>0$ and let the following two conditions hold:
\begin{itemize}
\item[(A)] there exists a non-empty subset $\Omega'$ of $\mbox{int}(V_{\omega_{2},\omega_{1}})$ which do have a cluster point in $\mbox{int}(V_{\omega_{2},\omega_{1}})$ and satisfies that, for every $z\in \Omega'$ and for every $j\in {\mathbb N}_{N},$ we have $Q_{j}(z)\in {\mathbb C}_{-};$  
\item[(B)] there exists $z\in \mbox{int}(V_{\omega_{2},\omega_{1}})$ such that, for every $j\in {\mathbb N}_{N},$ we have $Q_{j}(z)\in {\mathbb C}_{+}.$  
\end{itemize}
Then $\pm Q_{j}(B)h_{\mu}=\pm Q_{j}(\mu)h_{\mu},$
$e^{-(-B^{2})^{L}}h_{\mu}=e^{-(-\mu^{2})^{L}}h_{\mu},$ $\mu \in
\mbox{int}(V_{\omega_{2},\omega_{1}})$ and the operator $Q_{j}(B)$ is the integral generator of the
$C\equiv (e^{-(-z^{2})^{L}})(B)$-regularized semigroup $(W_{Q_{j}}(t)\equiv z\mapsto e^{tQ_{j}(z)}e^{-(-z^{2})^{L}})(B))_{t\geq 0}$
on $X$ ($j\in {\mathbb N}_{N}$). Furthermore, the set $R(C)$ is dense in $X.$ The validity of (A)-(B) yields that Theorem \ref{teoremica} and Remark \ref{dusty}(iii) can be applied, showing that
the $C$-regularized semigroups $((W_{Q_{j}}(t))_{t\geq 0})_{1\leq j\leq N}$ and
the operators
$e^{t_{1}Q(B)},\, e^{t_{2}Q(B)},\cdot \cdot \cdot, e^{t_{N}Q(B)}$ are densely disjoint distributionally chaotic. 
\end{example}

\begin{example}\label{cosine-functions} (cf. \cite[Example 2.13]{ejmaa})
Let us assume that $\zeta \geq 0,$ $-A\notin L(X)$, $-A$ generates an exponentially equicontinous $\zeta$-times integrated cosine function $(C_{\zeta}(t))_{t\geq 0}$, $N\in {\mathbb N},$ $N\geq 2$ and $P_{j}(z)=\sum_{i=0}^{n_{j}}a_{i,j}z^{i}$ is a non-zero complex polynomial with $a_{n_{j},j}>0$
($j\in {\mathbb N}_{N}$). Assume, further, that there are an open connected subset $\Omega$ of $\mathbb{C}$ and an analytic mapping $f:\Omega\to X\setminus \{0\}$ such that $\sigma_p(-A)\supseteq\Omega $ and
$f(\lambda)\in N(-A-\lambda)\setminus\{0\}$, $\lambda\!\in\!\Omega $ (e.g., let $a>0,$ let $\rho(x):=e^{-a|x|},$ $x\in {\mathbb R},$
$ X:=
L^{p}_{\rho}({\mathbb R}),$
$
D(B):=\{ f\in X \ | \ f(\cdot) \mbox{ is loc. abs. continuous, }\ f^{\prime} \in E \bigr\}$ and $Af:=f^{\prime},$ $f\in D(B);$ 
then $A$ generates $C_{0}$-group on $X$ and the above holds with $A=-B^{2},$ $\Omega =\{z^{2} : |\Re z|<a\}$ and $f (z^{2})=e^{z\cdot}$ for $|\Re z|<a;$ cf. \cite{fund} for the notion).

Set ${\mathcal A} :=\bigl(\begin{smallmatrix}0&I\\ -A &0\end{smallmatrix}\bigr),$ and suppose further that $\Omega'$ is a non-empty open connected subset of ${\mathbb C}$ such that $\lambda^{2}\in \Omega$ for all $\lambda \in {\Omega}'.$ Define $F : \Omega' \rightarrow (X\times X)\setminus \{(0,0)\}$
by $F(\lambda):=[f(\lambda^{2}) \ \lambda f(\lambda^{2})]^{T},$ $\lambda \in \Omega'.$ Then we know that $F(\cdot)$ is analytic, $\sigma_p({\mathcal A})\supseteq\Omega' $ and
$F(\lambda)\in N({\mathcal A}-\lambda)\setminus \{(0,0)\}$, $\lambda\!\in\!\Omega' .$ 
Further on, the operator
${\mathcal A}$ generates an exponentially equicontinuous $(\zeta+1)$-times integrated semigroup $(S_{\zeta+1}(t))_{t\geq 0}$
in $X\times X,$ which is given by
$$
S_{\zeta+1}(t):=\begin{pmatrix}
\int_0^tC_\zeta(s)\,ds &\int_0^t(t-s)C_{\zeta}(s)\,ds\\
C_\zeta(t)-g_{\zeta+1}(t)C\; &\int_0^tC_\zeta(s)\,ds\end{pmatrix},
\;\; t\geq 0.
$$
On the other hand, the operator ${\mathcal A}^{2}$ generates an exponentially equicontinuous, analytic $(\zeta/2)$-times integrated semigroup of angle  $\pi/2.$ Set $Q_{1}(z):=z$ and $Q_{j}(z):=-P_{j}(-z^{2})$ ($z\in {\mathbb C},$ $2\leq j\leq N$), as well as $A_{j}:=Q_{j}({\mathcal A}).$ Then the operator $A_{j}$ generates an exponentially equicontinuous, analytic $\eta$-times integrated semigroup $(S_{\eta}^{j}(t))_{t\geq 0}$ of angle $\pi/2,$ for $2\leq j\leq N$.
Suppose that the conditions (A) and (B) hold with the set $\mbox{int}(V_{\omega_{2},\omega_{1}})$ replaced by the set $\Omega'.$
These conditions ensure that Theorem \ref{teoremica} and Remark \ref{dusty}(iii) are applicable, so that
the integrated semigroups
$(S_{\zeta+1}(\cdot), (S_{\eta}^{j}(\cdot))_{2\leq j\leq N})$ are densely disjoint distributionally chaotic, which also holds for corresponding tuples of single operators.
\end{example}

Finally, at the end of this section, we would like to propose an interesting problem for our readers:

\begin{example}\label{enenminusjedan} (cf. also \cite[Example 3.1.35(i)]{knjigah},
\cite[Example 38]{filomat} and \cite[Example 5.12, Example 5.13]{mendoza}).
Let us assume that $n\in {\mathbb N},$ 
$\rho(t):=\frac{1}{t^{2n}+1},\ t\in {\mathbb R},$ $Af:=f^{\prime},$
$D(A):=\{f\in C_{0,\rho}({\mathbb R}) : f^{\prime} \in
C_{0,\rho}({\mathbb R})\},$ $X_{n}:=(C_{0,\rho}({\mathbb
R}))^{n+1},$ $D(A_{n}):=D(A)^{n+1}$ and $A_{n}(f_{1},\cdot \cdot
\cdot ,f_{n+1}):=(Af_{1}+Af_{2},Af_{2}+Af_{3},\cdot \cdot \cdot ,
Af_{n}+Af_{n+1},Af_{n+1}),$ $(f_{1},\cdot \cdot \cdot, f_{n+1}) \in
D(A_{n}).$ Then we already know that $\pm A_{n}$
generate global polynomially bounded $n$-times integrated semigroups
$(S_{n,\pm}(t))_{t\geq 0},$ and neither $A_{n}$ nor $-A_{n}$
generates a local $(n-1)$-times integrated semigroup. By \cite[Proposition 2.1.17]{knjigah}, the above implies that $A_{n}^{2}$ generates a polynomially bounded $n$-times integrated
cosine function $(C_{n}(t)\equiv1/2(S_{n,+}(t)+S_{n,-}(t)))_{t\geq 0}.$ Due to \cite[Corollary 2.4.9]{knjigah}, we have that $A_{n}^{2}$ generates a polynomially bounded $(n/2)$-times integrated
semigroup $(S_{n/2}(t))_{t\geq 0}$. We would like to ask whether $(S_{n/2}(t))_{t\geq 0}$ is densely disjoint distributionally chaotic and whether $(S_{n,\pm}(t))_{t\geq 0}$ and $(S_{n/2}(t))_{t\geq 0}$ are densely disjoint distributionally chaotic.
\end{example}

\section{Disjoint distributionally chaotic properties of abstract fractional PDEs}\label{d-MLOsz}

Let us recall that $\zeta \in (0,2)\setminus \{1\}.$ We start this section by providing analogues of Definition \ref{to fuck} and Definition \ref{dccfa} for fractional resolvent families:

\begin{defn}\label{to fuck-zer}
Let $\alpha_{j}\geq 0,$ let $C_{j}\in L(X)$ be injective for all $j\in {\mathbb N}_{N}$  and let $(R_{j}(t))_{t\geq
0}$ be a global $\zeta$-times $C_{j}$-regularized resolvent family with the integral generator $A_{j}$ ($j\in {\mathbb N}_{N}$). 
Suppose that $\tilde{X}$ is a closed linear
subspace of $X.$ Let $Z_{j,\zeta}(A_{j})$ the set
consisting of those vectors $x\in X$ such that
$R_{j}(t)x\in R(C_{j}),$ $t\geq 0$ and the mapping $t\mapsto
C^{-1}_{j}R_{j}(t)x,$ $t\geq 0$ is continuous. Denote by $t\mapsto C_{j}^{-1}R_{j}(t)x,$ $t\geq 0$ the unique mild
solution of the corresponding Cauchy problem \eqref{bez}, with the operator $A$ replaced by $A_{j}$ therein ($j\in {\mathbb N}_{N}$). Then we say that $((R_{j}(t))_{t\geq
0})_{1\leq j\leq N}$ are
disjoint
$\tilde{X}$-distributionally chaotic, $(d,\tilde{X})$-distributionally chaotic in short, iff there exist an uncountable
set $S\subseteq \bigcap_{j=1}^{N} Z_{j,\zeta}(A_{j}) \cap \tilde{X}$ and
$\sigma>0$ such that for each $\epsilon>0$ and for each pair $x,\
y\in S$ of distinct points we have that for each $j\in {\mathbb N}_{N}$ and $t\geq 0$ 
we have that
\begin{align*}
\begin{split}
& \overline{dens}\Biggl( \bigcap_{j\in {\mathbb N}_{N}} \bigl\{t\geq 0 :
d_{Y}\bigl(C^{-1}_{j}R_{j}(t)x,C^{-1}_{j}R_{j}(t)y\bigr)\geq \sigma \bigr\}\Biggr)=1,\mbox{ and }
\\ 
& \overline{dens}\Biggl( \bigcap_{j\in {\mathbb N}_{N}} \bigl\{t\geq 0 : d_{Y}\bigl(C^{-1}_{j}R_{j}(t)x,C^{-1}_{j}R_{j}(t)y\bigr)
<\epsilon \bigr\}\Biggr)=1.
\end{split}
\end{align*}

The sequence $((R_{j}(t))_{t\geq
0})_{1\leq j\leq N}$ is said to be densely
$(d,\tilde{X})$-distributionally chaotic iff $S$ can be chosen to be dense in $\tilde{X}.$
The set $S$ is said to be $(d,\sigma_{\tilde{X}})$-scrambled set\index{ $\sigma_{\tilde{X}}$-scrambled set} ($(d,\sigma)$-scrambled set in the case\index{$\sigma$-scrambled set}
that $\tilde{X}=X$) of the tuple $((R_{j}(t))_{t\geq
0})_{1\leq j\leq N}$;  in the case that
$\tilde{X}=X,$ then we also say that the sequence  $((R_{j}(t))_{t\geq
0})_{1\leq j\leq N}$ is (densely) disjoint distributionally chaotic, (densely) $d$-distributionally chaotic in short.
\end{defn}

\begin{defn}\label{dccfa-zer}
Let $\alpha_{j}\geq 0,$ let $C_{j}\in L(X)$ be injective for all $j\in {\mathbb N}_{N}$  and let $(R_{j}(t))_{t\geq
0}$ be a global $\zeta$-times $C_{j}$-regularized resolvent family with the integral generator $A_{j}$ ($j\in {\mathbb N}_{N}$). 
Suppose that $\tilde{X}$ is a closed linear
subspace of $X.$ Let $Z_{j,\zeta}(A_{j})$ be defined as above, and let $t\mapsto C_{j}^{-1}R_{j}(t)x,$ $t\geq 0$ be the unique mild
solution of the corresponding Cauchy problem \eqref{bez}, with the operator $A$ replaced by $A_{j}$ therein ($j\in {\mathbb N}_{N}$). Let $m\in {\mathbb N}$ and
$x\in  \bigcap_{j=1}^{N} Z_{j,\zeta}(A_{j})\cap \tilde{X}.$ Then we say that:
\begin{itemize}
\item[(i)] $x$ is disjoint distributionally near to $0$ for $((R_{j}(t))_{t\geq
0})_{1\leq j\leq N}$ iff
there exists $A\subseteq [0,\infty)$ such that $\overline{Dens}(A)=1$ and
$\lim_{s\rightarrow \infty,s\in A}C_{j}^{-1}R_{j}(s)x=0$ for all $j\in {\mathbb N}_{N};$ 
\item[(ii)] $x$ is disjoint distributionally $m$-unbounded for $((R_{j}(t))_{t\geq
0})_{1\leq j\leq N}$ iff
there exists a set $B\subseteq [0,\infty)$ satisfying that $\overline{Dens}(B)=1$ and
$$
\lim_{s\rightarrow \infty,s\in B}p_{m}\bigl(C_{j}^{-1}R_{j}(s)x\bigr)=0
$$ 
for all $j\in {\mathbb N}_{N};$ $x$ is
disjoint distributionally unbounded for the tuple $((R_{j}(t))_{t\geq
0})_{1\leq j\leq N}$ iff there exists $q\in {\mathbb N}$ such that
$x$ is disjoint distributionally $q$-unbounded for $((R_{j}(t))_{t\geq
0})_{1\leq j\leq N}$;
\item[(iii)] $x$ is a disjoint $\tilde{X}$-distributionally irregular vector for
$((R_{j}(t))_{t\geq
0})_{1\leq j\leq N}$
(disjoint distributionally irregular vector for $((R_{j}(t))_{t\geq
0})_{1\leq j\leq N}$ simply, in the case that
$\tilde{X}=X$) iff $x$ is both disjoint distributionally near to $0$ and
disjoint distributionally unbounded.
\end{itemize}
\end{defn}

Concerning disjoint distributional chaos of abstract time-fractional differential equations, the theoretical aspects are basically the same as for the abstract differential equations of first order and almost anything reasonable lies on possibal applications of Theorem \ref{cvddc}. Here we will formulate only one simple result regarding this theme:

\begin{thm}\label{sjajno}
Let $\alpha_{j}\geq 0,$ $t_{j}>0,$ let $C_{j}\in L(X)$ be injective for all $j\in {\mathbb N}_{N},$ and let $(R_{j}(t))_{t\geq
0}$ be a global $\zeta$-times $C_{j}$-regularized resolvent family with the integral generator $A_{j}$ ($j\in {\mathbb N}_{N}$). Let for each $i,\ j\in {\mathbb N}_{N}$ such that $i\neq j,$ we have $C_{i}A_{j}\subseteq A_{j}C_{i},$
$C_{i} R_{j}(t)=R_{j}(t)C_{i},$ $t\geq 0$ and $R_{j}(t)A_{i}\subseteq A_{i}R_{j}(t),$ $t\geq 0.$ Set $C:=\prod_{j=1}^{N}C_{j}.$ Then $({\mathrm R}_{j}(t)\equiv R_{j}(t)\prod_{1\leq i\leq N,i\neq j}C_{i})_{t\geq
0}$ is a global $\zeta$-times $C$-regularized resolvent family with the integral generator $A_{j}$ ($j\in {\mathbb N}_{N}$).
Suppose, further, that there exists a dense linear subspace $X_{0}$ of $X$ such that the following holds:
\begin{itemize}
\item[(a)] $\lim_{t\rightarrow \infty}{\mathrm R}_{j}(t)x=0,$ $x\in X_{0},$ $j\in {\mathbb N}_{N},$
\item[(b)] there exist $x\in X$ and $m\in {\mathbb N}$ such that
$\lim_{t\rightarrow \infty }p_{m}({\mathrm R}_{j}(t)x)=\infty$ for each $j\in {\mathbb N}_{N},$ resp.
$\lim_{t\rightarrow \infty }\|{\mathrm R}_{j}(t)x\|=\infty$ for each $j\in {\mathbb N}_{N},$ if $X$ is a Banach space.
\end{itemize}
Then $(({\mathrm R}_{j}(t))_{t\geq 0})_{1\leq j\leq N}$ and the operators $C^{-1}{\mathrm R}_{1}(t_{1}), C^{-1}{\mathrm R}_{2}(t),\cdot \cdot \cdot ,C^{-1}{\mathrm R}_{N}(t_{N})$ are disjoint distributionally chaotic; if, moreover, $R(C)$ is dense in $X,$
then\\ $(({\mathrm R}_{j}(t))_{t\geq 0})_{1\leq j\leq N}$ and the operators $C^{-1}{\mathrm R}_{1}(t_{1}), \, C^{-1}{\mathrm R}_{2}(t_{2}),\cdot \cdot \cdot ,\, C^{-1}{\mathrm R}_{N}(t_{N})$ are densely disjoint distributionally chaotic.
\end{thm}

\begin{proof}
Since for each $i,\ j\in {\mathbb N}_{N}$ such that $i\neq j,$ we have $C_{i}A_{j}\subseteq A_{j}C_{i},$
$C_{i} R_{j}(t)=R_{j}(t)C_{i},$ $t\geq 0$ and $R_{j}(t)A_{i}\subseteq A_{i}R_{j}(t),$ $t\geq 0,$ it follows immediately from definition that  $({\mathrm R}_{j}(t))_{t\geq
0}$ is a global $\zeta$-times $C$-regularized resolvent family with the integral generator $A_{j}$ ($j\in {\mathbb N}_{N}$), where $C$ is defined as above. Now the final conclusion follows from Theorem \ref{cvddc}, by considering the sequence $((C^{-1}{\mathrm R}_{j}(t))_{t\geq 0})_{1\leq j\leq N}$ of strongly continuous families consisting of linear continuous mappings between the spaces $[R(C)]$ and $X.$
\end{proof}

Now we will provide two illustrative applications of Theorem \ref{sjajno}, in which the regularizing operator $C$ is the identity operator (for general $C$, we can modify Example \ref{polinomi}; see also \cite[Example 2.5(iv)]{dynamo}):

\begin{example}\label{un}
(\cite{fund}, \cite{dynamo}) Let $a,\ b,\ c>0,$ $\zeta \in (1,2),$ $c<\frac{b^{2}}{2a}<1$ and
$$
\Lambda :=\Biggl\{ \lambda \in {\mathbb C} :
\Biggl|\lambda-c+\frac{b^{2}}{4a}\Biggr|\leq \frac{b^{2}}{4a},\ \Im (\lambda)
\neq 0 \mbox{ if }\Re (\lambda) \leq c-\frac{b^{2}}{4a} \Biggr\}.
$$
Then the operator $-A$ with domain $D(-A)=\{f\in
W^{2,2}([0,\infty)) : f(0)=0\},$ generates an analytic strongly
continuous semigroup of angle $\frac{\pi}{2}$ in the space
$X=L^{2}([0,\infty));$ the same holds in the case that the
operator $-A$ acts on $X=L^{1}([0,\infty))$ with domain
$D(-A)=\{f\in W^{2,1}([0,\infty)) : f(0)=0\}.$ In both cases, $-\Lambda \subseteq \sigma_{p}(A).$ Suppose that $\theta \in (\zeta \frac{\pi}{2}-\pi,\pi-\zeta
\frac{\pi}{2})$ and $P_{j}(z)=\sum \limits_{l=0}^{n}a_{l,j}z^{l}$ is a
non-constant complex polynomial such that $a_{l,n}>0$ and the following two conditions hold:
\begin{itemize}
\item[(A)'] there exists a non-empty subset $\Omega'$ of $-\Lambda$ which do have a cluster point in $-\Lambda$ and satisfies that, for every $z\in \Omega'$ and for every $j\in {\mathbb N}_{N},$ we have $-e^{i\theta}P_{j}(z)\notin \overline{\Sigma_{\zeta \pi/2}};$  
\item[(B)'] there exists $z\in -\Lambda$ such that, for every $j\in {\mathbb N}_{N},$ we have $-e^{i\theta}P_{j}(z)\in \Sigma_{\zeta \pi/2}.$  
\end{itemize}
We know that the
operator $-e^{i\theta}P_{j}(A)$ is the integral generator of an
exponentially bounded, analytic $\zeta$-times regularized resolvent
family $(R_{\zeta,\theta,P_{j}}(t))_{t\geq 0}$ of angle
$\frac{\pi-|\theta|}{\zeta}-\frac{\pi}{2}.$ Here,
the requirements needed for applying Theorem \ref{sjajno} are satisfied, which can be verified with the help of asymptotic expansion formuale \eqref{asim1}-\eqref{asim3} and the conditions (A)'-(B)'. As a consequence,
we have that $((R_{\zeta,\theta,P_{j}}(t))_{t\geq 0})_{1\leq j\leq N}$
are densely disjoint distributionally chaotic.
\item[(ii)] (\cite{ji}, \cite{dynamo}) Let $X$ be a symmetric
space of non-compact type and rank one, let $p>2,$ let the parabolic
domain $P_{p}$ and the positive real number $c_{p}$ possess the same
meaning as in \cite{ji}, and let $P^{j}(z)=\sum
\limits_{l=0}^{n}a_{l,j}z^{l},$ $z\in {\mathbb C}$ be a non-constant complex polynomial
with $a_{l,n}>0$ ($j\in {\mathbb N}_{N}$). Suppose that $\zeta \in (1,2),$
$
\pi-n\arctan
\frac{|p-2|}{2\sqrt{p-1}}-\zeta\frac{\pi}{2}>0
$
and
$
\theta \in
(n\arctan
\frac{|p-2|}{2\sqrt{p-1}}+\zeta\frac{\pi}{2}-\pi,\pi-n\arctan
\frac{|p-2|}{2\sqrt{p-1}}-\zeta\frac{\pi}{2}).
$
We know that
$-e^{i\theta}P^{j}(\Delta_{X,p}^{\natural})$ is the integral generator
of an exponentially bounded, analytic $\zeta$-times regularized
resolvent family $(R_{\zeta,\theta,P^{j}}(t))_{t\geq 0}$ of angle
$\frac{1}{\zeta}(\pi-n\arctan
\frac{|p-2|}{2\sqrt{p-1}}-\zeta\frac{\pi}{2}-|\theta|),$ for any $j\in {\mathbb N}_{N}.$ Using the fact that $\mbox{int}(P_{p})\subseteq
\sigma_{p}(\Delta_{X,p}^{\natural}),$ the validity of conditions (A)'-(B)' with the set $-\Lambda$ and polynomials $-e^{i\theta}P_{j}(z)$ replaced therein with the set int$(P_{p})$ and polynomials $-e^{i\theta}P^{j}(z)$
ensures that $(R_{\zeta,\theta,P^{j}}(t))_{t\geq 0}$ are densely disjoint distributionally chaotic.
\end{example}

We close the paper with the observation that distributionally chaotic properties of abstract multi-term fractional differential equations have been considered by the author in \cite{marek-nsjom}. 
Applying Theorem \ref{cvddc}, we can simply deduce several extensions of results established in this section for corresponding fractional resolvent operator families governing solutions of such equations.

In \cite{russmath} and \cite{zayed}, we have followed slightly different approaches to the concepts of disjoint hypercyclicity, disjoint topologically mixing property and the usual distributional chaos for abstract (multi-term) fractional differential equations. Disjoint
distributionally chaotic solutions of such equations can be analyzed by following this approach, as well. Related results will appear somewhere else.


\begin{thebibliography}{90}

\bibitem{angela}
Albanese, A.A., Barrachina, X., Mangino, E.M., Peris, A.,
Distributional chaos for strongly continuous semigroups of operators.
Commun. Pure Appl. Anal. 12 (2013), 2069--2082.

\bibitem{banasiakc}
Banasiak, J., Moszy\'nski, M.,
A generalization of Desch-Schappacher-Webb criterion for chaos.
Discrete Contin. Dyn. Syst. 12 (2005), 959--972.

\bibitem{bara}
Barrachina, X.,
Distributional Chaos of $C_{0}$-Semigroups of Operators.
Ph.D. Thesis, Universitat Polit\`echnica, Val\`encia, 2013.

\bibitem{bayart}
Bayart, F., Matheron, E.,
Dynamics of Linear Operators. Cambridge:
Cambridge Tracts in Mathematics, Cambridge University
Press, UK, 179(1) 2009.

\bibitem{bajlekova}
Bazhlekova, E.,
Fractional Evolution Equations in Banach
Spaces. Ph.D. Thesis, Eindhoven University of Technology,
Eindhoven, 2001.

\bibitem{2011}
Berm\'udez, T., Bonilla, A., Martinez-Gimenez, F., Peris, A.,
Li-Yorke and distributionally chaotic operators.
J. Math. Anal. Appl. 373 (2011),  83--93.

\bibitem{bg07}
Bernal-Gonz\'alez, L.,
Disjoint hypercyclic operators.
Studia Math. 182 (2007), 113-131.

\bibitem{2013JFA}
Bernardes Jr., N.C., Bonilla, A., M\"uler, V., Peris, A.,
Distributional chaos for linear operators. J. Funct. Anal.
265 (2013),  2143--2163.

\bibitem{2018JMAA}
Bernardes Jr., N.C., Bonilla, A., Peris, A., Wu, X.,
Distributional chaos for operators on Banach spaces.
J. Math. Anal. Appl. 459 (2018), 797--821.

\bibitem{bp07}
B\`es, J., Peris, A.,
Disjointness in hypercyclicity.
J. Math. Anal. Appl. 336 (2007), 297-315.

\bibitem{bm201345}
B\`es, J., Martin, \"O., Peris, A., Shkarin, S.,
Disjoint mixing operators.
J. Funct. Anal. 263 (2013), 1283--1322.

\bibitem{ejmaa}
Chen, C.-C., Kosti\' c, M., Pilipovi\'c, S., Velinov, D.,
$d$-Hypercyclic and $d$-chaotic properties of abstract differential equations of first order.
Electronic J. Math. Anal. Appl. 6 (2018), 1--26.

\bibitem{chen-chen}
Chen, C.-C., Chen, K.-Y., Kosti\' c, M.,
Distributional chaos for weighted translation operators on groups.
preprint, https://arxiv.org/pdf/1807.05191.

\bibitem{mendozaa}
Conejero, J.A., Lizama, C., Murillo-Arcdila, M., Peris, A.,
Linear dynamics of semigroups generated
by differential operators. Open Math. {15} (2017), 
745--767.

\bibitem{mendoza}
Conejero, J.A., Kosti\' c, M., Miana, P.J., Murillo-Arcila, M.,
Distributionally chaotic families of operators on Fr\' echet spaces.
Comm. Pure Appl. Anal. 15 (2016), 1915--1939.

\bibitem{ralphchaos}
deLaubenfels, R., Emamirad, H., Grosse--Erdmann, K.-G.,
Chaos for semigroups of unbounded operators.
Math. Nachr. 261/262 (2003), 47--59.

\bibitem{fund}
Desch, W., Schappacher, W.,  Webb, G.F., Hypercyclic
and chaotic semigroups of linear operators. Ergodic Theory Dynam. Systems
17 (1997), 1--27.

\bibitem{duan}
Duan, J., Fu, X.-C., Liu, P.-D., Manning, A.,
A linear chaotic quantum harmonic oscillator.
Appl. Math. Lett. 12 (1999), 15--19.

\bibitem{russmath}
Fedorov, V., Kosti\'c, M.,
Disjoint hypercyclic and disjoint topologically mixing properties of degenerate fractional differential equations. Russian Math. 61 (2018), 31--46. 

\bibitem{erdper}
Grosse-Erdmann, K.-G., Peris, A.,
Linear Chaos. London: Springer-Verlag, London 2011.

\bibitem{ji} 
Ji, L., Weber, A., Dynamics of the heat semigroup on symmetric
spaces. Ergod. Th. Dynam. Sys. 30 (2010), 457--468.

\bibitem{knjigah}
Kosti\'c, M., Generalized Semigroups and Cosine Functions.
Mathematical Institute SANU, Belgrade 2011.

\bibitem{knjigaho}
Kosti\'c, M., Abstract Volterra Integro-Differential
Equations. Boca Raton, New York and London: Taylor and Francis Group/CRC Press, Boca Raton, Fl. 2015.

\bibitem{FKP}
Kosti\' c, M., 
Abstract Degenerate Volterra Integro-Differential Equations: Linear Theory and Applications.
Mathematical Institute SANU, Belgrade, accepted.

\bibitem{marek-nsjom}
Kosti\' c, M., 
Distributionally chaotic properties of abstract fractional differential equations.
Novi Sad J. Math. 45 (2015), 201--213.

\bibitem{dynamo}
Kosti\' c, M., 
Hypercyclicity and topologically mixing property for abstract time-fractional equations. 
Dyn. Syst. 27
(2012), 213--221.

\bibitem{filomat}
Kosti\' c, M., 
Hypercyclic and chaotic integrated $C$-cosine functions, Filomat 26 (2012), 1--44.

\bibitem{zayed}
Kosti\' c, M.,
The existence of distributional chaos in abstract degenerate fractional differential equations. 
J. Fract. Calc. Appl. 7 (2016), 153--174.

\bibitem{marek-trio}
Kosti\'c, M.,
Distributionally chaotic multivalued linear operators and their generalizations.
preprint.

\bibitem{revista}
Kosti\' c, M.,
Disjoint distributional chaos in Fr\' echet spaces.
preprint.

\bibitem{zhengtaiwan} 
Li, M., Zheng, Q., Zhang, J., Regularized resolvent
families. Taiwanese J. Math. 11 (2007), 117--133.

\bibitem{ma10} 
Martin, \"O., Disjoint Hypercyclic and Supercyclic Composition Operators. 
Ph.D. Thesis, Bowling Green State University, 2010.

\bibitem{p-oprocha-tams}
Oprocha, P.,
Distributional chaos revisited. 
Trans. Amer. Math. Soc. 361 (2009), 4901--4925.

\bibitem{p-oprocha}
Oprocha, P.,
A quantum harmonic oscillator and strong chaos. 
J. Phys. A 39 (2006), 14559--14565.

\bibitem{prus}
Pr\"uss, J., 
Evolutionary Integral Equations and
Applications. Basel:
Birkh\"auser-Verlag, 1993.

\bibitem{smital}
Schweizer, B., Sm\' ital, J.,
Measures of chaos and a spectral decomposition of dynamical systems on the interval.
Trans. Amer. Math. Soc. 344 (1994), 737--754.

\end{thebibliography}
\end{document}